\documentclass[12pt,leqno]{article}
\usepackage{amsfonts}
\usepackage{amsmath, amsthm, amsfonts, amssymb, color, cite}
\usepackage{mathrsfs}
\usepackage{color}
\newtheorem{theorem}{Theorem}[section]

\newtheorem{lemma}[theorem]{Lemma}
\newtheorem{proposition}[theorem]{Proposition}

\newtheorem{remark}[theorem]{Remark}
\theoremstyle{definition}

\newcommand{\scr}[1]{\mathscr #1}
\definecolor{wco}{rgb}{0.5,0.2,0.3}

\numberwithin{equation}{section} \theoremstyle{remark}

\newcommand{\ua}{\uparrow}

\DeclareMathOperator{\Lip}{Lip}

\DeclareMathOperator{\supp}{supp}

\title{{\bf   Wasserstein asymptotics for empirical measures of diffusions on four dimensional closed manifolds  }}
\author{{\bf Dario Trevisan$^{a)}$ \, \  Feng-Yu Wang$^{b)}$ \, \ Jie-Xiang Zhu$^{c)}$    }\\
\footnotesize{$^{a)}$ Dipartimento di Matematica, Universit\`a di Pisa, Italy}\\
\footnotesize{$^{b)}$ Center for Applied Mathematics, Tianjin University, China}\\
\footnotesize{$^{c)}$ Department of Mathematics, Shanghai Normal University, China}\\
\footnotesize{ dario.trevisan@unipi.it, \, \    wangfy@tju.edu.cn, \, \ jiexiangzhu7@gmail.com}}
\begin{document}
\allowdisplaybreaks
\def\R{\mathbb R}  \def\ff{\frac} \def\ss{\sqrt} \def\B{\mathbf
B}
\def\N{\mathbb N} \def\kk{\kappa} \def\m{{\bf m}}
\def\ee{\varepsilon}\def\ddd{D^*}
\def\dd{\delta} \def\DD{\Delta} \def\vv{\varepsilon} \def\rr{\rho}
\def\<{\langle} \def\>{\rangle}
  \def\nn{\nabla} \def\pp{\partial} \def\E{\mathbb E}
\def\d{\text{\rm{d}}} \def\bb{\beta} \def\aa{\alpha} \def\D{\scr D}
  \def\si{\sigma} \def\ess{\text{\rm{ess}}}\def\s{{\bf s}}
\def\beg{\begin} \def\beq{\begin{equation}}  \def\F{\scr F}
\def\Ric{\mathcal Ric} \def\Hess{\text{\rm{Hess}}}
\def\e{\text{\rm{e}}} \def\ua{\underline a} \def\OO{\Omega}  \def\oo{\omega}
 \def\tt{\tilde}\def\[{\lfloor} \def\]{\rfloor}
\def\cut{\text{\rm{cut}}} \def\P{\mathbb P} \def\ifn{I_n(f^{\bigotimes n})}
\def\C{\scr C}      \def\aaa{\mathbf{r}}     \def\r{r}
\def\gap{\text{\rm{gap}}} \def\prr{\pi_{{\bf m},\nu}}  \def\r{\mathbf r}
\def\Z{\mathbb Z} \def\vrr{\nu} \def\ll{\lambda}
\def\L{\scr L}\def\Tt{\tt} \def\TT{\tt}\def\II{\mathbb I}
\def\i{{\rm in}}\def\Sect{{\rm Sect}}  \def\H{\mathbb H}
\def\M{\mathbb M}\def\Q{\mathbb Q} \def\texto{\text{o}} \def\LL{\Lambda}
\def\Rank{{\rm Rank}} \def\B{\scr B} \def\i{{\rm i}} \def\HR{\hat{\R}^d}
\def\to{\rightarrow} \def\gg{\gamma}
\def\EE{\scr E} \def\W{\mathbb W}
\def\A{\scr A} \def\Lip{{\rm Lip}}\def\S{\mathbb S}
\def\BB{\scr B}\def\Ent{{\rm Ent}} \def\i{{\rm i}}\def\itparallel{{\it\parallel}}
\def\g{{\mathbf g}}\def\Sect{{\mathcal Sec}}\def\T{\mathcal T}\def\BB{{\bf B}}
\def\f\ell \def\g{\mathbf g}\def\BL{{\bf L}}  \def\BG{{\mathbb G}}
\def\Bd{{D^E}} \def\BdP{D^E_\phi} \def\Bdd{{\bf \dd}} \def\Bs{{\bf s}} \def\GA{\scr A}
\def\Bg{{\bf g}}  \def\Bdd{\psi_B} \def\supp{{\rm supp}}\def\div{{\rm div}}
\def\ddiv{{\rm div}}\def\osc{{\bf osc}}\def\1{{\bf 1}}\def\BD{\mathbb D}
\def\H{{\bf H}}\def\gg{\gamma} \def\n{{\mathbf n}}\def\GG{\Gamma}\def\HAT{\hat}
\def\SU{{\bf SU}}\def\Var{{\rm Var}}\def\DD{\Delta}\def\T{\mathbb T}
\def\eps{\varepsilon}
\newcommand{\bra}[1]{\left( #1 \right)}
\newcommand{\dbra}[1]{\left \langle #1 \right |}
\newcommand{\dket}[1]{\left| #1 \right \rangle}
\newcommand{\sprod}[2]{\left< #1 | #2 \right>}
\newcommand{\sqa}[1]{\left[ #1 \right]}
\newcommand{\cur}[1]{\left\{ #1 \right\}}
\newcommand{\ang}[1]{\left< #1 \right>}
\newcommand{\abs}[1]{\left| #1 \right|}
\newcommand{\nor}[1]{\left\| #1 \right\|}
\newcommand{\dist}{\mathsf{d}_M}
\newcommand{\les}{\lesssim}
\newcommand{\vol}{\operatorname{vol}}
\newcommand{\cost}{\mathsf{c}}
\newcommand{\V}{\mathbf{V}}

\maketitle

\begin{abstract}     We identify the leading term in the asymptotics of the quadratic Wasserstein distance between the invariant measure and empirical measures for diffusion processes on closed weighted four-dimensional Riemannian manifolds. Unlike results in lower dimensions, our analysis shows that this term depends solely on the Riemannian volume of the manifold, remaining unaffected by the potential and vector field in the diffusion generator.

  \end{abstract} \noindent

 \vskip 2cm
\section{Introduction and main result}

On a $d$-dimensional closed Riemannian manifold $M$, let $L:=\DD+\nn V \nabla +Z$,
$$L(f) = \DD f + \nn V \, \nn f + Z f, \quad \forall f \in C^2(M),$$
where $\DD$ is the Laplacian, $V\in C^2(M)$ such that for the Riemannian volume measure ${\vol}$
$$\mu(\d x):=\e^{V(x)}{\rm vol}(\d x)$$ is a probability measure,
and $Z$ is a $C^1$-vector field (i.e., a derivation) with ${\rm div}_\mu(Z)=0$:
$$\int_M Z f  \d \mu =0,\ \ \  \forall f\in C^1(M).$$
Let $X := (X_t)_{t\ge 0}$ be the diffusion process generated by $L$ and write $P_t$ for its transition semigroup. It is well-known that $X$ is exponentially ergodic with $\mu$ as its unique invariant probability measure. Consider  then the empirical measure
$$\mu_T:= \ff 1 T\int_0^T \dd_{X_t}\d t,\ \ T>0,$$ where for any $x\in M$, $\dd_x$ denotes the Dirac measure at point $x$. By ergodicity, a.s.\ the weak convergence of probabilities $\mu_T \to \mu$ holds as $T \to \infty$.  It is of interest to establish quantitative convergence results in terms of suitable metrics on the space of probabilities on $M$. A natural choice is provided here by the Wasserstein distance $\W_p$ (for any $p \ge 1$) induced by the Riemannian distance $\rr$ on $M$. The distance $\W_p$ is the $p$-th root of the optimal cost required to transport $\mu_T$ into $\mu$, where the displacement cost from $x$ to $y$ is given by the $p$-th power of the Riemannian distance $\rr(x,y)^p$. In this context, it has been deeply investigated in a series of recent works  \cite{wang2022wasserstein, wang2023convergence, wang2023limit} and the following asymptotic behavior holds \cite[theorem 1.1]{wang2024sharp}, given  $1 \le p\le \max\cur{\frac{2d}{(d-2)^+},  \frac{d(d-2)}2 }$, for every $x \in M$:
\begin{equation}\label{eq:rates}
\E^x\sqa{ \W_p^p(\mu_T, \mu)} \sim\footnote{ The notation $f(T)\sim g(T)$ means that it holds $c^{-1} \le f(T)/g(T) \le c$  for $T$ sufficiently large, for some constant $c \in (0, \infty)$ possibly depending on $M$, $L$ and other parameters but not $T$. We also collect for later use the notation $f(T) \les g(T)$, when it holds $f(T) \le c g(T)$ for $T$ sufficiently large, for some constant $c < \infty$.} \begin{cases} T^{-p/2}& \text{if $d  \le 3$,}\\
 \bra{(\log T)/T }^{p/2} & \text{if  $d=4$,}\\
T^{-p/(d-2)} & \text{if $d \ge 5$,}
\end{cases}
\end{equation}
where $\E^x$ denotes the expectation with respect to the probability $\P^x$ under which the diffusion process has initial condition $X_0 = x$. It is conjectured that \eqref{eq:rates} can be extended to all $1 \leq p < \infty$.

If one denotes with $R_{p,d}(T)$ the right hand side in \eqref{eq:rates}, the existence of the limit
\begin{equation}\label{eq:limit-general}
 \lim_{T \to \infty}  \frac{\E^x\sqa{ \W_p^p(\mu_T, \mu)}}{R_{p,d}(T)} =: \cost(L,p)
\end{equation}
is also naturally conjectured, although it is only proved so far in the case $d \le 3$ in \cite{wang2023convergence, wang2024sharp}. For $p=2$, it reads
\begin{equation}\label{eq:limit-d-le-3}
  \lim_{T \to \infty} T \E^x\sqa{ \W_2^2(\mu_T, \mu)}  = \sum_{i=1}^\infty   \ff 2 {\ll_i^2} \bra{1-\ff 1{\ll_i} {\V}(Z \phi_i)},
\end{equation}
where  $\{\ll_i\}_{i\ge 1}$ are the strictly positive eigenvalues of $-(\DD+\nn V)$ in $L^2(\mu)$, $\{\phi_i\}_{i\ge 1}$ are associated  unit $L^2(\mu)$-norm eigenfunctions and $\V$  denotes the quadratic form
\begin{equation}\label{eq:v-def}\V(\phi):= \int_0^\infty \int_{M} \phi P_t \phi \, \d \mu \, \d t.\end{equation}
 The appearance of $\V$ is ultimately due to the central limit theorem  in this setting \cite{liming1995moderate}: for every $\phi \in L^2(\mu)$ with $\int_M \phi \, \d \mu = 0$, one has convergence in law, as $T \to \infty$,
\begin{equation}\label{eq:CLT}\frac{1}{\sqrt{T}} \int \phi \d \mu_T = \ff 1 {\ss T}\int_0^T \phi(X_t)  \d t\to \mathcal N (0, 2{\V}(\phi)).\end{equation}
Let us notice that, since $\V(\phi) \ge 0$ for every $\phi$, the limit \eqref{eq:limit-d-le-3} yields that $\cost(L,2) \le \cost(\Delta +\nabla V \nabla , 2)$, i.e., convergence is faster (although with the same asymptotic rate) in the non-symmetric case, i.e., if $Z \neq 0$.

If the dimension of $M$ is larger than $3$, existence of the limit \eqref{eq:limit-general} is currently an entirely open problem. In \cite[section 1.3]{mariani2023wasserstein}, it is conjectured that for $d \ge 5$, the constant $\cost(\Delta + \nabla V \nabla , p)$ could be given by an expression in terms of the corresponding limiting constant for the Brownian interlacement occupation measure, although only an upper bound for $M = \T^d$ the flat torus and $V = 0$ is established  \cite[theorem 1.2]{mariani2023wasserstein}.
%

In this communication we show the validity of \eqref{eq:limit-general} for $d=4$ and $p=2$.

\begin{theorem}\label{T1}
With the notation introduced above, for a closed Riemannian manifold $M$ with dimension $d=4$, weighted volume measure $\mu$, and the occupation measure $\mu_T$ of the diffusion process with generator $L = \Delta + \nabla V \nabla + Z$, it holds 
\begin{equation}\label{eq:lim-exp} \sup_{x\in M}\bigg| \ff T{\log T}\E^x\sqa{ \W_2^2 \bra{ \mu_T, \mu } }- \frac{\vol(M)}{8 \pi^2 } \bigg|
\les  \sqrt{ \frac{\log\log T}{\log T}}. \end{equation}
\end{theorem}

Thus, we explicitly compute that $\cost(L, 2) = \vol(M)/(8 \pi^2)$ for any $4$-dimensional closed Riemannian manifold. In particular, we show that the leading term in the Wasserstein asymptotics does not depend on $V$, nor $Z$. In particular, the independence from the field $Z$ can be explained as follows: in \eqref{eq:limit-d-le-3}, the series of the terms $1/\lambda_i^2$  diverges logaritmically when $d=4$, but the series of terms $\V(Z\phi_i)/\lambda_i^3$ is still convergent (see \eqref{eq:v-z-phi-conv} below). This phenomenon, although novel in this setting, is not completely unexpected, for in the transport of i.i.d.\ samples on weighted two-dimensional manifolds  \cite{AST, AmGlau, benedetto2021random, ambrosio2022quadratic} the leading term in the asymptotics also depends on the volume only.

The overall structure of the proof borrows from the literature of transport of i.i.d.\ samples, and in particular \cite{AmGlau}: in Section \ref{sec:mod}, we prove Theorem \ref{T2} concerning the asymptotics for the transport cost  between a smoothed empirical measure $\mu_{T,\vv}:= \hat P_\vv^*\mu_T  $ for $\vv>0$, where $\hat P_\vv $ is the symmetric diffusion semigroup generated by $\hat L:=\DD+\nn V \nn$, and $\hat{P}_\vv^*$ denotes its dual action on measures.  
%
Next, in Section \ref{sec:refinedcontractivity}, we provide an estimate on the expectation of  ${ \W_2^2(\mu_T, \mu_{T,\vv}) }$, which refines the simpler contraction estimate
\beq\label{TR} \W_2^2(\mu_T,\mu_{T,\vv}) \le c \vv,\ \ T>0, \, \vv\in (0,1)\end{equation}
valid for some constant $c = c(M)>0$. We acknowledge that both these facts parallel the argument from \cite{AmGlau}, but they require new considerations and the introduction of different tools than the case of transport of i.i.d.\ points.  Finally, in Section \ref{sec:proof-main}, we combine them to establish our main result.

Unless otherwise stated, we always take in what follows $d=4$ and set $\vv:= (\log T)^\gamma / T$ for some fixed constant $\gg>3$ and consider the case of $X$ being stationary, i.e., any marginal law of $X$ equals $\mu$. To keep notation simple we write $\mathbb{P}$ and $\E$ for expectation with respect to this law.

\section{Asymptotics for the smoothed empirical measures}\label{sec:mod}


\beg{theorem}\label{T2}
With the notation introduced above, on a four dimensional closed Riemannian manifold $M$,  it holds
$$\bigg|\ff{T}{\log T}\E[\W_2^2(\mu_{T,\vv},\mu)]- \ff{{\rm vol}(M)}{8\pi^2}\bigg|\les  \frac{\log \log T}{\log T}.$$
\end{theorem}

To prove this result, we  rely on a well-known bound for $\W_2$ in terms of the $H^{-1}$ norm of the density $u_{T,\vv}$ of the smoothed empirical measure $\mu_{T,\vv}$ with respect to $\mu$.

We introduce first some notation.  Let $\hat p_t$ be the (symmetric) heat kernel of $\hat P_t$ with respect to $\mu$. Then, we have
\beq\label{MT1}\mu_{T,\vv}= \hat P_\vv^*\mu_T=: u_{T,\vv}\mu, \quad  { u_{T,\vv}}:= \ff 1 T\int_0^T \hat p_\vv(X_t,\cdot)\d t.\end{equation}
Next, we consider the Poisson kernel
$$q_\vv(x,y):= \int_0^\infty [\hat p_{t+\vv}(x,y)-1]\d t,\ \ x,y\in M,$$
so that
\beq\label{MT2} f_{T,\vv}:= (-\hat L)^{-1}(u_{T,\vv}-1)= \ff 1 T\int_0^T q_\vv(X_t,\cdot)\d t.\end{equation}
In the following, we first estimate the expectation of $\mu(|\nn f_{T,\vv}|^2) := \int_M |\nn f_{T,\vv}|^2 \d \mu$, then prove the above theorem by comparing
it with $\E[\W_2^2(\mu_{T,\vv},\mu)]$.

\subsection{Estimate of $\E \sqa{ \mu( |\nn f_{T,\vv}|^2)}$} Recall that, for $i\ge1$, $\phi_i$  denotes a (zero mean) unit norm eigenfunction (for $\hat{L}$) with eigenvalue $-\ll_i<0$, i.e. $\|\phi_i\|_{L^2(\mu)}=1$ and $\hat L\phi_i=-\ll_i\phi_i.$ We define
 $$\psi_i(T):= \ff 1 {\ss T} \int_M \phi_i \d \mu_T = \ff 1 {\ss T} \int_0^T \phi_i(X_t) \d t,\ \ T>0,$$
where the factor $1/ \ss T$ arises from the central limit theorem \eqref{eq:CLT}. 
By \eqref{MT1} and the spectral representation $\hat p_\vv(x,y) -1 = \sum_{i=1}^\infty \e^{-\lambda_i \eps} \phi_i(x) \phi_i(y)$, we have
\begin{equation} \label{spectral}
u_{T, \eps} - 1 = \ff{1}{\ss T} \sum_{i = 1}^{\infty} \e^{-\ll_i \vv} \psi_i(T) \phi_i.
\end{equation}
Therefore,
\begin{equation*}
f_{T, \eps} = (-\hat L)^{-1}(u_{T, \eps} - 1) = \ff{1}{\ss T} \sum_{i = 1}^{\infty} \ff{\e^{-\ll_i \vv}}{\ll_i} \psi_i(T) \phi_i.
\end{equation*}
Since $(\phi_i)_{i \geq 1}$ is an orthonormal sequence in $L^2(\mu)$, we find after an integration by parts
\beq\label{MT3}
\E \sqa{ \mu \bra{ |\nn f_{T,\vv}|^2 } } =  \E\sqa{ \int_M  f_{T, \eps} (-\hat{L}) f_{T,\vv} \d \mu }  = \ff{1}{T} \sum_{i = 1}^{\infty} \ff{\e^{- 2\ll_i \vv}}{\ll_i} \E[|\psi_i(T)|^2].
\end{equation}

We claim that the following expansion for $\E \big[ |\psi_i(T)|^2 \big]$ holds (we keep the dimension $d$ general, for possible future reference).

\begin{proposition}\label{lemma:ito-tanaka-eigen} For any closed Riemannian manifold $M$ with dimension $d\ge 1$,
there exists a constant $c=c(M, L)>0$ such that
\begin{equation} \label{eq:estimate-e-lambda}
\Big|\E \big[ |\psi_i(T)|^2 \big]-\ff 2 {\ll_i}+ \ff 2 {\ll_i^2} {\bf V}(Z\phi_i)\Big| \le \ff c {\ll_i (1+T)},\ \ i\ge 1, \, T>0.
\end{equation}
\end{proposition}

Using this in our four-dimensional setting we deduce the main bound for this part:
\beq\label{FN0} \bigg|\ff{T}{\log T}\E \sqa{ \mu\bra{  |\nn f_{T,\vv}|^2 }}-
\ff{{\rm vol}(M)}{8\pi^2}\bigg|\les \ff {\log \log T}{\log T}.\end{equation}
Indeed,  combining $\|\nn \phi_i\|_{L^2(\mu)}=\ss{\ll_i}$  with the gradient estimate from \cite[Lemma 3.1]{wang2023convergence}, we find that for some constants $c=c(M, L)<\infty$ and $ \ll =\ll(M, L)> 0$, it holds
$$ \| \nn P_t  f \|_{L^2(\mu)} \le   c (1 \wedge t)^{-\frac12} \e^{-\ll t} \| f \|_{L^2(\mu)}, \ \ t > 0, f\in L^2(\mu),$$ and we have, integrating by parts and using that $\div_\mu Z = 0$,
$$ \big| \mu((Z\phi_i)P_t (Z\phi_i)) \big| = \big| \mu(\phi_i ZP_t (Z\phi_i)) \big| \le 
c \|Z\|_\infty^2 (1 \wedge t)^{-\frac12} \e^{-\ll t} \ss{\ll_i}.$$
So, ${\bf V}(Z\phi_i)\le c_1 \ss{\ll_i}$ for some constant $c_1 = c_1(M,L)< \infty$. Next, we recall the small time asymptotics for the heat trace \cite[Corollary 3.2.]{charalambous2019heat}:
\begin{align} \label{small time trace}
\sum_{i = 1}^{\infty} \e^{-t \ll_i} = {\rm tr} \, \e^{t \hat L} - 1 = \frac{\vol(M)}{16 \pi^2 t^2} + O( t^{-1} ), \quad \text{ as   $t \to 0$.}
\end{align}
By standard Tauberian arguments we find for the eigenvalue counting function $N(\lambda) := \sum_{i=1}^\infty 1_{ \cur{ \lambda_i \le \lambda}}$ the asymptotics $\lambda^{-2} N(\lambda)  = \vol(M)/ (32 \pi^2) + o(\lambda^{-2})$, which yields convergence of the series
\begin{equation}\label{eq:v-z-phi-conv} \sum_{i=1}^\infty \ff{1}{\ll_i^3} {\bf V}(Z\phi_i) \le c_1 \sum_{i=1}^\infty \ll_i^{-5/2} = \frac {5 c_1}{2} \int_0^\infty \lambda^{-5/2 -1} N(\lambda) \d \lambda <\infty;
\end{equation}
as well as the asymptotics
\begin{equation}\label{Asy}
\sum_{i=1}^\infty \ff{\e^{-2 \eps \ll_i}}{\ll_i^2} =\ff{{\rm vol}(M)}{16 \pi^2 } \log  (\eps^{-1}) + O(1),  \quad \text{as $\eps\to 0$,}
\end{equation}
which can be also obtained directly from \eqref{small time trace} and integrating by parts. Indeed,
\beq\label{AX} \sum_{i=1}^\infty\ff{\e^{-\ll_i s}}{\ll_i}= \int_s^1 \sum_{i=1}^\infty \e^{-\ll_i t}\d t + O(1)
 = \ff{{\rm vol}(M)}{16 \pi^2 s} + O(\log s^{-1}) \quad \text{as $s\to 0$},\end{equation} and
 \begin{equation*}\begin{split} &\sum_{i=1}^\infty \ff{\e^{-2\vv \ll_i}}{\ll_i^2} = \int_{2\vv}^\infty \sum_{i=1}^\infty\ff{\e^{-\ll_i s}}{\ll_i} \d s
 = \int_{2\vv}^1 \sum_{i=1}^\infty\ff{\e^{-\ll_i s}}{\ll_i} \d s  + O(1)\\
 &= \int_{2\vv}^1 \Big[\ff{{\rm vol}(M)}{16 \pi^2 s}+ O(\log s^{-1})\Big]\d s + O(1)
  = \ff{{\rm vol}(M)}{16 \pi^2 } \log  (\vv^{-1}) + O(1), \quad \text{as}\ \vv \to 0.\end{split}
  \end{equation*}
  Combining the claim \eqref{eq:estimate-e-lambda} with these estimates in \eqref{eq:estimate-e-lambda}
yields \eqref{FN0}. 

\begin{proof}[Proof of Proposition \ref{lemma:ito-tanaka-eigen}]
By \cite[Lemma 4.2(1)]{wang2023convergence}, we have the identity
\beq\label{*0}
{\bf V}(\phi_i)= \ff 1 {\ll_i}- \ff{1}{\ll_i^2}{\bf V}(Z\phi_i).
\end{equation}
By the Markov property and stationarity of $X$, it follows that (writing $\mu(g) = \int_{M} g \d \mu$)
\beq\label{*1}
\beg{split} \E \big[ |\psi_i(T)|^2 \big] &= \ff 2 T\int_0^{T} \d t_1 \int_{t_1}^T \mu(\phi_i P_{t_2 - t_1} \phi_i) \d t_2 =  \ff 2 T\int_0^{T} \d t_1 \int_{0}^{T - t_1} \mu(\phi_i P_t \phi_i) \d t\\
 &= 2\Big(\ff 1 {\ll_i}-\ff 1{\ll_i^2} {\bf V}(Z \phi_i)\Big) - \ff 2 T \int_0^T \d t_1 \int_{T - t_1}^\infty \mu(\phi_i P_t \phi_i)\d t.
\end{split}
\end{equation}
We next evaluate the integrand $\mu(\phi_i P_t \phi_i)$. By Duhamel's formula
 \beq \label{eq:Duhamel}
P_t f = \hat P_t f + \int_{0}^{t} P_s (Z \hat P_{t - s} f) \d s, \ \ t \geq 0
\end{equation}
and $\phi_i$ being an eigenfunction for $\hat{L}$, so that $\hat P_t \phi_i = \e^{-\ll_i t} \phi_i$, we have
\beq\label{ZZ}
P_t \phi_i = \e^{-\ll_i t} \phi_i + \int_0^t \e^{-\ll_i(t-s)} P_s (Z\phi_i)\d s,
\end{equation}
and therefore
\beq\label{*2}
\mu(\phi_i P_t\phi_i) =\e^{-\ll_i t} +\int_0^t \e^{-\ll_i (t-s)} \mu\big((P_s^*\phi_i) Z\phi_i\big) \d s,
\end{equation}
where we write $P_t^*$ for the semigroup with generator $L^* : = \hat L - Z$.
Noting that $\mu(\phi_i Z \phi_i) = 0$ and using Duhamel's formula \eqref{eq:Duhamel} again
$$P_s^* \phi_i = e^{-\ll_i s} \phi_i - \int_0^s \e^{-\ll_i(s-r)} P_r^* (Z\phi_i)\d r,$$
we obtain
$$\mu\big((P_s^*\phi_i) Z\phi_i\big) =-\int_0^s \e^{-\ll_i(s-r)}\mu\big((Z\phi_i)P_r(Z\phi_i)\big)\d r.$$
Combining above identities with the fact that
$$ \big| \mu\big((Z\phi_i)P_r(Z\phi_i)\big) \big| \le \e^{-\ll_1 r}  \|Z\phi_i\|_{L^2(\mu)}^2\le \|Z\|_\infty^2 \ll_i \e^{-\ll_1 r},$$
we derive
\beg{align*}
\big| \mu(\phi_i P_t\phi_i) - \e^{-\ll_i t} \big| \le   \|Z\|_\infty^2 \ll_i\int_0^t \e^{-\ll_i(t-s)}\d s \int_0^s \e^{-\ll_i(s-r)-\ll_1r}\d r \le C \ll_i^{-1} \e^{-\ll_1 t/2}.
\end{align*}
Combining with \eqref{*1} we finish the proof.
\end{proof}

\subsection{Estimate on $\E[|\W_2^2(\mu_{T,\vv},\mu)-\mu(|\nn f_{T,\vv}|^2)|]$  }

The main result in this part is the following (with the assumptions $d=4$ and on $\eps$).

\beg{proposition}\label{PER}
It holds
\begin{equation} \label{eq: main result1}
\E\sqa{\left|  \W_2 \left(\mu_{T,\vv}, \mu \right) - \sqrt{ \mu(|\nn f_{T,\vv}|^2) } \right|} \les \sqrt{\frac{1}{T\log T}},
\end{equation}
and
\begin{equation} \label{eq: main result2}
\E\sqa{\left|  \W_2^2\left(\mu_{T,\vv}, \mu \right) -  \mu(|\nn f_{T,\vv}|^2)  \right|} \les \frac{1}{T}.
\end{equation}
\end{proposition}

To prove this result, we need some lemmas. Let $\hat L_x$ and $\nn^n_y$, for $n\in \mathbb N$ stand for the corresponding operators acting on variables $x,y\in M$ respectively (let $\nn^0$ be the identity operator). We state and prove the following result for general dimension $d$ and $\eps \in (0,1)$, for possible future reference.

\beg{lemma}\label{L1}
For any $n\in \mathbb N$ and $p\in (1,\infty)$, there exists a constant $c = (n,p,M,V)< \infty$ such that, for any $\vv \in (0,1)$,
\begin{equation}\label{eq:mixed-bound}
\sup_{y \in M} \int_M \big|\nn_x\hat L_x^{-1} \nn_y^n q_\vv(x,y) \big|^p \mu(\d x) \le c \beg{cases} \vv^{-\ff{(d+n-3)p-d}2}, & \textrm{if } (d + n - 3) p > d;\\
 \{\log (1+\vv^{-1}) \big \}^p, &\textrm{if } (d + n - 3) p = d;\\
 1, &\textrm{if } (d + n - 3) p < d.\end{cases}
\end{equation}
and
\begin{equation}\label{eq:norm-bound}
\sup_{y \in M} \int_M \big|\nn_y^n q_\vv(x,y) \big|^p \mu(\d x) \le c \beg{cases} \vv^{-\ff{(d + n - 2)p-d}2}, & \textrm{if } (d + n - 2) p > d;\\
 \{\log (1+\vv^{-1}) \big \}^p, &\textrm{if } (d + n - 2) p = d;\\
 1, &\textrm{if } (d + n - 2) p < d.\end{cases}
\end{equation}
\end{lemma}

\begin{proof}
These bounds could be established by the same argument in \cite[Corollary 3.13]{AmGlau} using pointwise upper bound of $\big|\nn_x\hat L_x^{-1} \nn_y^n q_\vv \big|$. Here we provide an alternative approach. Let $\rr$ be the Riemannian distance on $M$, so that the standard heat kernel bounds give,  for some constants $c \in (1, \infty)$ and $\ll>0$,
\beq \label{0}
\big| \nn_y^n \big( \hat p_t(x,y)-1 \big) \big|\le c t^{-\ff{d+n}2} \e^{-\ll t- \frac{\rr(x,y)^2}{c t}}, \ \ x, y \in M, \, t>0.
\end{equation}
Consequently, given $1 < p < \infty$, for every $y \in M$ and $t > 0$,
\begin{align*}
& \int_M |\nn_y^n \big( \hat p_t(x,y)-1 \big)|^p \mu(\d x)\\
&\leq c_1 t^{-\ff{(d+n)p}2} \bigg( \int_{\{ x \, ; \, \rr(x, y) \leq \sqrt{t}  \}} + \sum_{k = 1}^{\infty} \int_{\{ x \, ; \,  k \sqrt{t} < \rr(x, y) \leq  (k+1) \sqrt{t}  \}} \bigg)  \e^{- p \ll t - \frac{p \rr(x,y)^2}{c t}} \mu(\d x)\\
&  \leq c t^{-\ff{(d+n)p}2 + \frac{d}{2}} \e^{- p \ll t} \big( 1 + \sum_{k = 1}^{\infty} (k + 1)^d \e^{-\frac{p k^2}{c}} \big) \leq c t^{-\ff{(d+n)p}2 + \frac{d}{2}} \e^{- p \ll t},
\end{align*}
where we conventionally keep denoting with $c$ possibly different constants, and we also used the fact that $\sup_{y \in M} \mu(\{  x \, ; \, \rr(x, y) \leq r  \}) \le c r^d$ for some (possibly different) constant $c<\infty$. Then,
\begin{equation}\label{1}
\sup_{y \in M} \big\| \nn_y^n \big( \hat p_t(\cdot,y) - 1 \big) \big\|_{L^p(\mu)} \le c t^{-\frac{d + n}{2} + \frac{d}{2p}} \e^{- \ll t}.
\end{equation}

On the other hand, notice that
$$\nn_x (-\hat L_x)^{-1} \nn_y^n \hat p_t(x,y)= \nn_x (-\hat L_x)^{-\ff 1 2} \nn_y^n (-\hat L_x)^{-\ff 1 2} \hat  p_t(x,y)= \nn_x (-\hat L_x)^{-\ff 1 2} \nn_y^n (-\hat L_y)^{-\ff 1 2} \hat  p_t(x,y).$$
By the definition of $q_\vv$ and the $L^p$-boundedness of the Riesz transform, we find a constant $c$ such that 
\beq\label{2} I:= \big\|\nn (-\hat L)^{-1} \nn_y^n q_\vv(\cdot,y)\big\|_{L^p(\mu)} \le c \big\|\nn_y^n  (-\hat L_y)^{-\ff 1 2} q_\vv(\cdot,y)\big\|_{L^p(\mu)}.\end{equation}
Since
$$(-\hat L_y)^{-\ff 1 2} q_\vv(x,y) = \ff 1{\ss\pi} \int_0^\infty \ff 1 {\ss s} \hat P_t q_\vv(x,\cdot)(y) \d t = \ff 1 {\ss\pi} \int_0^\infty \d t\int_0^\infty \ff{\hat p_{t+s+\vv}(x,y)-1}{\ss s} \d s,$$
\eqref{1} and \eqref{2} yield
\beg{align*}
I &\le c \int_0^\infty \d t\int_0^\infty \ff{\big\| \nn_y^n \big( \hat p_{t + s + \vv}(\cdot,y) - 1 \big) \big\|_{L^p(\mu)} }{\ss s} \d s\\
&\le c \int_0^\infty s^{-\ff 1 2}\d s\int_0^\infty \e^{-\ll( t + s + \vv)} (t+s+\vv)^{-\frac{d + n}{2} + \frac{d}{2p}}  \d t\\
&\le c \int_0^\infty s^{-\ff 1 2}\e^{-\ll s}  \Big\{  (s+\vv)^{-\ff{[(d+n-2)p-d]^+}{2p}}+   1_{\{(d+n-2)p=d\}} \log \{1+ (s+\vv)^{-1}\} \Big\}\d s.  \end{align*}
This implies the desired estimate \eqref{eq:mixed-bound}. Similarly,
\beg{align*}
\big\|\nn_y^n q_\vv\big\|_{L^p(\mu)} \leq \int_{\vv}^{\infty} \big\| \nn_y^n \big( \hat p_t(\cdot,y) - 1 \big) \big\|_{L^p(\mu)} \d t.
\end{align*}
Combined this with \eqref{1}, the proof of \eqref{eq:norm-bound} is complete.
\end{proof}

\begin{remark}
It is easy to see that \eqref{0} also implies that for any  $n\in \mathbb N$ such that $d + n > 2$,
\begin{align} \label{eq: derivative bound}
\big\| \nn^n f_{T, \eps}  \big\|_\infty \leq c \int_{\vv}^{\infty} \sup_{x, y \in M} \big| \nn_y^n \big( \hat p_t(x,y)-1 \big) \big| \d t \les \eps^{-\frac{d + n -2}{2}}.
\end{align}
\end{remark}

The second step towards the proof of Proposition \ref{PER} is to evaluate the probability of the event 
\begin{equation}\label{eq:event-A} A_{T, \eps}^{\xi} := \cur{ \nor{ \nabla ^2 f_{T, \eps}  }_\infty \le \xi },
\end{equation}
for $\xi>0$. 
To this aim, we collect the following concentration inequality for diffusion processes, see also \cite[Corollary 3.2]{wang2024sharp}.

\begin{lemma} \label{concentration}
Assume that the dimension of $M$ is $d \geq 3$. Then, there exists a constant $c = c(M,L) \in (0, \infty)$ such that, for every $g \in L^{d/2}(M)$ with zero mean (i.e., $\mu(g) = 0$) and $T, \xi > 0$,
\begin{align} \label{bernstein}
\mathbb{P} \bra{  \left| \frac1T \int_{0}^{T} g(X_t) \d t \right| > \xi } &\leq 2 \exp \bigg(  - \frac{2 T \xi^2}{\sigma^2(g) \big( \sqrt{1 + 2 c \| g \|_{L^{d/2}(\mu)} \xi/\sigma^2(g)} + 1 \big)^2} \bigg) \nonumber\\
& \leq 2 \exp \bigg(  - \frac{T \xi^2}{2 \big( \sigma^2(g) \, + \, c \| g \|_{L^{d/2}(\mu)} \xi \big)} \bigg),
\end{align}
where
$$
\sigma^2( g ) := 2 \int_M g (- \hat L)^{-1} g \, \d \mu =  \int_M \bigg|\nn \int_0^\infty \hat P_t g\bigg|^2 \d\mu.
$$
\end{lemma}
\begin{proof}
By \cite[theorem 1]{wu2000deviation} we have
\beq\label{WP}
\mathbb{P} \bra{ \left| \frac1T \int_{0}^{T} g(X_t) \d t \right| > \xi } \leq 2 \exp \big( - T I_g(\xi-) \big), \quad \forall \, T, \xi > 0,
\end{equation}
where
$$
I_g(\xi):= \inf \big\{\mu(|\nn h|^2) \, ; \, h\in W^{2,1}(\mu), \mu(h^2) = 1, \, |\mu(h^2g)| = \xi, \, \mu(h^2 |g|) < \infty \big\}
$$
and
$$
\quad I_g(\xi -) : = \lim_{\varepsilon \to 0+} I_g(\xi - \varepsilon).
$$
Following to the argument used in \cite[theorem 2.2]{gao2014bernstein}, for every $h \in W^{2,1}(\mu)$ with $\mu(h^2) = 1$, we notice that
$$
\frac{2 |\mu(g h^2)|^2}{\sigma^2(g) \big( \sqrt{1 + 2 c \| g \|_{L^{d/2}(\mu)} |\mu(g h^2)| /\sigma^2(g)} + 1 \big)^2} \leq \mu(|\nn h|^2)
$$
is equivalent to
\begin{equation}\label{XX}
|\mu(g h^2)| \leq \sqrt{2\sigma^2(g) \mu(|\nn h|^2)} +  c \| g \|_{L^{d/2}(\mu)} \mu(|\nn h|^2).
\end{equation}
So, by \eqref{WP}, it suffices to verify \eqref{XX} for some constant $c>0$.
We write
\begin{align*}
\mu(g h^2) = 2 \bar{h} \mu(g h) + \mu\big( g (h - \bar h)^2  \big),
\end{align*}
where $\bar{h}: = \mu(h)$. Since $|\bar{h}| \leq \mu(h^2)^{1/2} = 1$ and by Cauchy-Schwarz inequality
\begin{align*}
|\mu(g h)| = | \mu \big(  (-\hat L)^{-\frac12} g \cdot (-\hat L)^{\frac12} h \big)| \leq \sqrt{\frac{\sigma^2(g) \mu(|\nn h|^2)}{2}},
\end{align*}
we obtain
$$
|2 \bar{h} \mu(h g)| \leq \sqrt{2 \sigma^2(g) \mu(|\nn h|^2)}.
$$
On the other hand, by the Sobolev-Poincar\'e inequality on $M$, 
it follows that $h \in L^\frac{2d}{d - 2}(\mu)$ and there exists $c <\infty$ such that
$$
\| h - \bar{h} \|_{L^\frac{2d}{d - 2}(\mu)} \leq c \mu(|\nn h|^2)^{1/2}.
$$
Combined with H\"older's inequality, it follows that
\begin{align*}
\abs{\mu\big( g (h - \bar h)^2 \big)} \leq  \| g \|_{L^{d/2}(\mu)} \| h - \bar{h} \|_{L^\frac{2d}{d - 2}(\mu)}^2  \leq c^2 \| g \|_{L^{d/2}(\mu)} \mu(|\nn h|^2),
\end{align*}
which implies \eqref{XX} up to replacing $c$ with $c^2$. 
\end{proof}

Back to our four-dimensional setting, we apply the concentration inequality to estimate the probability of $A^{\xi}_{T, \vv}$ with $\xi = 1/(\log T)$, which is sufficient for our purposes.

\begin{lemma} \label{prop:flatness}
There exists a constant $C = C(M, L, \gamma)>0$ such that, for $\xi = 1/\log T$, it holds
$$ \mathbb{P}  \big(  ( A_{T, \eps}^\xi )^{\mathsf{c}} \big) \les  \exp\big(-C(\log T)^{\gamma-2}\big). $$
\end{lemma}

\begin{proof}
For fixed $y \in M$, applying   \eqref{bernstein} with $g = \nabla^2_y q_{\eps}(\cdot, y)$, and
using Lemma \ref{L1} with $p=n=2$ and $d=4$,  we find a constant $C>0$ such that
\begin{align*}
\mathbb{P} \bra{  \abs{ \nabla^2_y f_{T, \eps}(y)} > \xi/2 } &\les \exp \bigg( - \frac{T \xi^2}{2   \int_M|\nabla_x \hat L_x^{-1} \nabla^2_y q_{\eps}(x, y) |^2  \mu (\d x) + c \| \nabla^2_y q_{\eps}(\cdot, y) \|_{L^2(\mu)} \xi }   \bigg)\\
& \les \exp \big( - C (\log T)^{\gamma-2}  \big).
\end{align*}
Furthermore, by \eqref{eq: derivative bound} with $n = 3$, we can always bound the Lipschitz constant of $y \mapsto |\nabla^2 f_{T,\eps}|(y)$ in terms of
$$ K : =  \big\| \nn^3 f_{T, \eps}  \big\|_\infty \les \eps^{-\frac{5}{2}}.$$
Thus, choosing a suitable $\ell$-net with $K \cdot \ell = \xi/2$, hence with $N(\ell) \les (K / \xi)^4$ elements, we obtain that
$$
\mathbb{P} \bra{ \sup_{y\in M} \abs{ \nabla^2_y f_{T, \eps}(y)} > \xi } \les N(\ell)  \cdot  \exp ( - C (\log T)^{\gamma - 2} \big) \les  \eps^{-10} \xi^{-4} \exp ( - C (\log T)^{\gamma - 2}  ).
$$
This implies the desired estimate, for a smaller constant $C>0$.
\end{proof}

\beg{proof}[Proof of Proposition \ref{PER}] We set $\xi = 1/\log T$ and consider the event $A_{T, \eps}^{\xi}$. By Lemma \ref{prop:flatness}, \eqref{eq: main result1} follows if
\beq\label{Le2} \E \sqa{1_{A_{T,\eps}^\xi}\left|  \W_2 \left(\mu_{T,\vv}, \mu \right) - \sqrt{ \mu(|\nn f_{T,\vv}|^2) } \right|} \les \sqrt{\frac{1}{T\log T}}.\end{equation}
To prove this estimate, we introduce the probability measure $ \hat \mu_{T,\eps} := \exp(\nabla f_{T, \eps})_\# \mu$.  By \cite[theorem 1.1]{glaudo2019c} on $A^{\xi}_{T,\eps}$ the map  $\nabla f_{T, \eps}$ is the optimal map transforming from $\mu$ to $\hat{\mu}_{T,\eps}$, so that
\beq\label{LM}
\W_2^2 \bra{\hat{\mu}_{T,\eps}, \mu} = \mu(|\nn f_{T,\vv}|^2).
\end{equation}
Next, we argue that, still on $A^{\xi}_{T,\eps}$
\beq\label{LM2}
\W_2^2 \bra{\mu_{T, \eps}, \hat\mu_{T,\eps}} \les \xi^2 \mu(|\nn f_{T,\vv}|^2).
\end{equation}
This is a consequence of the Dacorogna-Moser interpolation scheme: since $f_{T,\eps}=(-\hat L)^{-1} (u_{T,\eps}-1)$ and $\div_\mu  \circ \nn = \hat L,$ where, the function $u_s := (1 - s) + s u_{T, \eps}$ and the time-dependent vector field
$$
Y_s := \frac{\nabla f_{T, \eps}}{u_s},\ \ s\in [0,1]
$$
satisfy the equation
$$
\frac{\d}{\d s}u_s + \text{div}_{\mu}\big( u_s Y_s  \big) = 0.
$$
Then, by \cite[Proposition A.1.]{AmGlau}, one obtains \eqref{LM2}.

By the triangle inequality, we derive
\beq\label{Le3}
\E \sqa{1_{A_{T,\eps}^\xi}\left|  { \W_2 \left(\mu_{T, \eps}, \mu \right)} - \sqrt{ \mu(|\nn f_{T,\vv}|^2) } \right|} \les   \xi \E \sqa{\mu(|\nn f_{T,\vv}|^2)}^{1/2}.
\end{equation}
This together with \eqref{FN0} implies \eqref{Le2},   and hence \eqref{eq: main result1} is proved.

To prove the other estimate,  write
\begin{align*}
\E \sqa{\left|  \W_2^2\left( \mu_{T,\vv}, \mu \right) -  \mu( | \nabla f_{T, \eps} |^2 )  \right|} \leq \big(\E \sqa{ A^2 }\big)^{1/2} \Big(\big( \E \sqa{ A^2 }\big)^{1/2} + 2 \big(\E\sqa{B^2}\big)^{1/2} \Big),
\end{align*}
where
$$
A = \W_2 \left(\mu_{T,\vv}, \mu \right) - \sqrt{ \mu(|\nn f_{T,\vv}|^2) }, \quad B = \sqrt{\mu(|\nn f_{T,\vv}|^2) }.
$$
Noting that \eqref{LM} and the triangle inequality imply
\begin{align*}
\E \sqa{ A^2 } &= \E \sqa{ \bigg| \W_2 ( \mu_{T, \eps}, \mu ) - \sqrt{ \mu(|\nn f_{T,\vv}|^2) } \bigg|^2  }\\
& \leq   \E \sqa{\W_2^2\big( \mu_{T, \eps}, \hat \mu_{T,\eps} \big)}\\
&\le D^2\P\big((A_{T,\vv}^\xi)^c\big)+ \E\big[1_{A_{T,\vv}^\xi} \W_2^2\big( \mu_{T, \eps}, \hat \mu_{T,\eps} \big)\big]\\
& \les \xi^2 \mu(|\nn f_{T,\vv}|^2) \les \frac{1}{T \log T},
\end{align*}
where $D$ is the diameter of $M$. Then \eqref{eq: main result2}  follows from    \eqref{FN0}, Lemma \ref{prop:flatness} and \eqref{LM2}.

\end{proof}

 \section{Improved contractivity estimate} \label{sec:refinedcontractivity}


Aim of this section is to establish the following improved version of \eqref{TR}.

\beg{lemma}\label{LLN}
With the notation introduced above, on a four dimensional closed Riemannian manifold $M$,  it holds
\beq
\label{AX2}
\E \left[  \W_2^2  \left(\mu_T, \mu_{T, \eps} \right) \right] \les  \frac{\log \log T}{T}.
\end{equation}
\end{lemma}

\beg{proof} Let $\xi= {1}/{\log T}$. Then on the event $A_{T,\eps}^\xi$ introduced in \eqref{eq:event-A}, we have
$$\|u_{T,\eps}-1\|_\infty = \|\hat L f_{T,\eps}\|_\infty\les \|\nabla^2 f_{T,\eps}\|_\infty\le \xi =1/(\log T)$$
where we have used the fact that $| \nn V \nn f_{T,\eps} | \les \|\nn V \|_\infty \|\nn^2 f_{T,\eps}\|_\infty.$
Combining this with
Ledoux's upper bounds for $\W_2$ (see \cite{AmGlau, Le17} or \cite[Lemma A.1]{wang2022wasserstein}), we have for large $T>0$ and any $\eps' \in (0, \eps)$,
$$\W_2^2(\mu_{T, \eps'}, \mu_{T, \eps})\le 4  \int_M  \ff{|\nabla (-\hat L)^{-1} (u_{T, \eps'} - u_{T, \eps})|^2}{u_{T, \eps}}  \d \mu \leq 8 \int_M |\nabla (-\hat L)^{-1} (u_{T, \eps'} - u_{T, \eps})|^2 \d \mu.$$
Then, on $A^{\xi}_{T, \vv}$ we find
\beq\label{Le1}  \beg{split}
\W_2^2(\mu_{T, \eps'}, \mu_{T, \eps}) & \les   \int_M  |\nabla (-\hat L)^{-1} (u_{T, \eps'} - u_{T, \eps})|^2  \d \mu
=  \frac{1}{T}\sum_{i = 1}^{\infty} \frac{(\e^{- \lambda_i \eps'} - \e^{- \lambda_i \eps})^2}{\ll_i}   |\psi_i(T)|^2 \\
& \leq \frac{1}{T} \sum_{i = 1}^{\infty} \frac{\e^{- 2 \lambda_i \eps'} - \e^{- 2 \lambda_i \eps}}{\lambda_i}  |\psi_i(T)|^2 = 2 \int_{\eps'}^{\eps} \| u_{T, s} - 1 \|_{L^2(\mu)}^2  \d s.
\end{split}
\end{equation}
Next, by \eqref{spectral}, Lemma \ref{lemma:ito-tanaka-eigen} and \eqref{AX}, we have
$$\E \sqa{ \| u_{T, s} - 1 \|_{L^2(\mu)}^2 } = \frac{1}{T} \sum_{i = 1}^{\infty} \e^{-2 \lambda_i s} \E \left[ |\psi_i(T)|^2 \right] \les T^{-1} \sum_{i=1}^{\infty} \ll_i^{-1} \e^{-2 \lambda_i s} \les T^{-1} s^{-1}.$$
 By \eqref{Le1}, this gives
$$
\E \big[ 1_{A_{T,\eps}^\xi} \W_2^2(\mu_{T, \eps'}, \mu_{T, \eps}) \big]\les T^{-1} \log(\eps/\eps').
$$
Then, using the triangle inequality and the fact that $\E\big[ \W_2^2(\mu_T, \mu_{T, \eps'})  \big]\les \eps'$ by \eqref{TR}, we derive
\begin{align*}
\E\left[ 1_{A_{T, \eps}^{\xi}} \W_2^2  \left(\mu_T, \mu_{T, \eps} \right) \right] &\les \E\big[ \W_2^2(\mu_T, \mu_{T, \eps'})  \big] + \E\big[ 1_{A_{T, \eps}^{\xi}} \W_2^2(\mu_{T, \eps'}, \mu_{T, \eps}) \big]\\
& \les \eps' + T^{-1} \log(\eps/\eps').
\end{align*}
Choosing finally $\eps' = \frac{\log \log T}{T}$ and combining with Proposition \ref{prop:flatness}, we finish the proof.
\end{proof}

\section{Proof of theorem \ref{T1}}\label{sec:proof-main}
 We consider first  the case of $X$ being stationary. In this situation, we argue that
 \beq\label{ST} \bigg|\ff T{\log T}\E[\W_2^2(\mu_T,\mu)] - \ff{{\rm vol}(M)}{8\pi^2}\bigg|\les
 \ss{\ff{\log\log T}{\log T}}.\end{equation}
Indeed, by \eqref{FN0} and Proposition \ref{PER} we have
$$ \bigg|\ff T{\log T}\E[\W_2^2(\mu_{T,\vv},\mu)]-\ff{{\rm vol}(M)}{8\pi^2}\bigg|\les
\ff {\log \log T}{\log T}.$$
Combining this with Lemma \ref{LLN} implies
\beg{align*} &\bigg|\ff T{\log T}\E[\W_2^2(\mu_T,\mu)] - \ff{{\rm vol}(M)}{8\pi^2}\bigg|\\
&\les \ff {\log \log T}{\log T}+\ff T{\log T} \Big|\E[\W_2^2(\mu_T,\mu)-\W_2^2(\mu_{T,\vv},\mu)]\Big| \\
&\le \ff {\log \log T}{\log T}+ \ff T{\log T}\Big(\E[\W_2^2(\mu_T, \mu_{T,\vv})] +2 \sqrt{ \E[\W_2^2(\mu_T, \mu_{T,\vv})] \cdot \E[\W_2^2(\mu_T, \mu)] }\Big)\\
&\les \ff {\log \log T}{\log T} + \ss{\ff{\log\log T}{\log T} }
\les  \ss{\ff{\log\log T}{\log T} }.\end{align*}
So, \eqref{ST} holds in the case of $X$ stationary.

To address the general case, for any $x\in M$, let $(X_t^x,X_t^\mu)$ be the coupling by reflection for the diffusions generated by $L$ with initial distributions $\dd_x$ and $\mu$ respectively (so that $X^\mu$ is stationary). According to the proof of  \cite[Theorem 1]{chen1997general}, there exists an increasing function $g: [0,D]\to [0,\infty)$, where $D$ is the diameter of $M$, such that
$$c_1r\le g(r)\le c_2 r,\ \ r\in [0,D]$$
holds for some constants $c_2>c_1>0$ (independent of $x$) and that
$$\d g(\rr(X_t^x,X_t^\mu))\le \d M_t-\dd g(\rr(X_t^x,X_t^\mu))\d t$$ holds for some martingale $M_t$ and a constant $\dd>0$. Therefore by taking expectation, we obtain
$$
\E [ g(\rr(X_t^x, X_t^\mu))] \leq \mu \big( g(\rho(x, \cdot)) \big) \e^{-\delta t}.
$$
Consequently,  writing $$\mu_T^x:=\ff 1 T\int_0^T\dd_{X_t^x}\d t \quad \text{ and }\quad   \mu_T^\mu:= \ff 1 T\int_0^T \dd_{X_t^\mu}\d t$$ satisfy
$$\E[\W_2^2(\mu_T^x,\mu_T^\mu)]\le \ff 1 T\int_0^T \E[\rr^2(X_t^x,X_t^\mu)]\d t\le \ff{c_1^{-1} c_2 D^2}{T}\int_0^T\e^{-\dd t}\d t\le \ff{c_1^{-1} c_2 D^2}{\dd T}.$$
We thus derive  from \eqref{ST} that
\beg{align*} &\sup_{x\in M}\big|\E^x[\W_2^2(\mu_T,\mu)] -\E[\W_2^2(\mu_T^{\mu},\mu)]\big|= \sup_{x\in M}\big|\E[\W_2^2(\mu_t^x,\mu)- \W_2^2(\mu_T^\mu,\mu)]\big|\\
&\le \sup_{x\in M}\Big( \E [\W_2^2(\mu_T^x,\mu_T^\mu)]+ 2 \E\big[\W_2(\mu_T^x,\mu_T^\mu)\W_2(\mu_T^\mu,\mu)\big]\Big)\\
&\les \ff 1 T + \ff 1 {\ss T} \sqrt{\E[\W_2^2(\mu_T^{\mu},\mu)]} \les \ff{\ss{\log T}}T.\end{align*}
Applying \eqref{ST} again, we derive that for large $T$,
\beg{align*} \sup_{x\in M} \bigg|\ff T{\log T}\E^x[\W_2^2(\mu_{T,\vv},\mu)]-\ff{{\rm vol}(M)}{8\pi^2}\bigg|\les
 \ss{\ff{\log\log T}{\log T}} +  \ff{1}{\ss{\log T}}\les \ss{\ff{\log\log T}{\log T}},\end{align*}
 and the proof is completed.


\bibliographystyle{amsplain}
\bibliography{OT.bib}

\bigskip

\textbf{Acknowledgements.}\\
\indent D.T.\ acknowledges the MUR Excellence Department Project awarded to the Department of Mathematics, University of Pisa, CUP I57G22000700001,  the HPC Italian National Centre for HPC, Big Data and Quantum Computing - Proposal code CN1 CN00000013, CUP I53C22000690001, the PRIN 2022 Italian grant 2022WHZ5XH - ``understanding the LEarning process of QUantum Neural networks (LeQun)'', CUP J53D23003890006, the INdAM-GNAMPA project 2024 ``Tecniche analitiche e probabilistiche in informazione quantistica'' and the project  G24-202 ``Variational methods for geometric and optimal matching problems'' funded by Università Italo Francese.  Research also partly funded by PNRR - M4C2 - Investimento 1.3, Partenariato Esteso PE00000013 - "FAIR - Future Artificial Intelligence Research" - Spoke 1 "Human-centered AI", funded by the European Commission under the NextGeneration EU programme.

\sloppy F.-Y.W.  and J.-X.Z.  acknowledge  the National Key R\&D Program of China (No. 2022YFA1006000, 2020YFA0712900) and NNSFC (11921001).

\end{document}